\documentclass[11pt]{amsart}
\pdfoutput=1
\usepackage{amsmath, amssymb}
\usepackage{amsfonts}
\usepackage{mathrsfs}
\usepackage[arrow,matrix,curve,cmtip,ps]{xy}
\usepackage{graphicx}
\usepackage{float}
\usepackage{amsthm}

\allowdisplaybreaks

\newtheorem{theorem}{Theorem}[section]
\newtheorem{lemma}[theorem]{Lemma}
\newtheorem{proposition}[theorem]{Proposition}

\newtheorem*{theorem*}{Theorem}
\theoremstyle{remark}

\newtheorem{example}[theorem]{Example}
\newtheorem{question}{Question}

\numberwithin{equation}{section}

\begin{document}
\title[Collections of Pairwise Intersecting Curves]{ Large Collections of Curves Pairwise Intersecting Exactly Once}

\author{Tarik Aougab}

\address{Department of Mathematics \\ Yale University \\ 10 Hillhouse Avenue, New Haven, CT 06510 \\ USA}
\email{tarik.aougab@yale.edu}

\date{\today}



\begin{abstract}

Let $\Omega=(\omega_{j})_{j\in I}$ be a collection of pairwise non-isotopic simple closed curves on the closed, orientable, genus $g$ surface $S_{g}$, such that $\omega_{i}$ and $\omega_{j}$ intersect exactly once for $i\neq j$.  It was recently demonstrated by Malestein, Rivin, and Theran that the cardinality of such a collection is no more than $2g+1$. In this paper, we show that for $g\geq 3$, there exists atleast two such collections with this maximum size up to the action of the mapping class group, answering a question posed by Malestein, Rivin and Theran. 
\end{abstract}

\maketitle

\section{Introduction}
In surface topology, it has long been of general interest to pose and study problems related to the existence and enumeration of collections of curves on a surface satisfying specified intersection properties. Specifically, the following question, popularized by Benson Farb and which goes back to work of Juvan, Malnic and Mohar \cite{Juv-Mal-Mo}, is an example of such a problem. In all that follows, let $S_{g}$ denote the closed orientable surface of genus $g$. \vspace{2 mm }

\begin{question} Let $\Gamma=(\gamma_{j})_{j\in I}$ be a collection of pairwise non-isotopic simple closed curves on $S_{g}$ indexed by an index set $I$ such that $\gamma_{i}$ intersects $\gamma_{j}$ no more than once for any $i \neq j$. What is the maximum cardinality of $I$? 

\end{question}

This turns out to be a very difficult question; at the time of the writing this paper, all that is so far known is that $I$ must be finite, and that there exists a quadratic lower bound and an exponential upper bound for the maximum size of $I$, the best of which are due to the work of Malestein, Rivin and Theran \cite{Mal-Riv-Ther} . Following the same terminology of their paper, we define a $k$-$\textit{system}$ to be a collection of simple closed curves on $S_{g}$ such that no two intersect more than $k$ times. 
 In the same article, these authors address the following simpler variant of Question 1: \vspace{3 mm}

\begin{question} Let $\Omega=(\omega_{j})_{j\in I}$ be a collection of pairwise non-isotopic simple closed curves indexed by an index set $I$ such that $\gamma_{i}$ intersects $\gamma_{j}$ $\textit{exactly}$ once for any $i \neq j$. What is the maximum cardinality of $I$? 

\end{question}

Given any $1$-system, we can construct its $\textit{intersection graph}$, which has a vertex for each curve in the collection, and an edge between any two curves that intersect. The $1$-systems which are the objects of study in Question 2 are precisely those $1$-systems whose intersection graphs are complete; therefore, we call such collections $\textit{complete 1-systems}$. Malestein, Rivin and Theran answer Question 2 by showing that any complete $1$-system can have no more than $2g+1$ members, and that this bound is sharp for all $g\geq 1$.  

\begin{example} (Genus $1$) Recall that two simple closed curves are isotopic on the torus (denoted by $S_{1}$) if and only if they are homologous; therefore the distinct isotopy classes of simple closed curves are parameterized by ordered pairs of coprime integers. Given two such classes $\omega_{1}=(a,b)$ and $\omega_{2}=(c,d)$ recall further that the $\textit{geometric intersection number}$ $i(\omega_{1},\omega_{2})\in \mathbb{N}$, defined as the minimum number of times that any curve in $\omega_{1}$ intersects a curve in $\omega_{2}$, is given by 
\[ i(\omega_{1},\omega_{2})=  \left| \det \left( \begin{array}{cc}
  																	a & c \\
  																	b & d \end{array} \right) \right| \]
Therefore by basic linear algebra, if we are given two curves $\omega_{1},\omega_{2}$ on $S_{1}$ which intersect once, there is exactly one curve up to isotopy intersecting both $\omega_{1}$ and $\omega_{2}$ exactly once. Hence any complete $1$-system on $S_{1}$ has no more than $3$ members.

\end{example}

By a $\textit{maximal complete 1-system}$, we mean a complete $1$-system on $S_{g}$ containing $2g+1$ members. Recall that the $\textit{mapping class group}$ of a surface $S$, denoted $\mbox{Mod}(S)$, is the group of isotopy classes of orientation preserving homeomorphisms from $S$ to itself (for more details, see \cite{Far-Mar}), and that $\mbox{Mod}(S_{1})$ is isomorphic to $SL(2,\mathbb{Z})$. Example 1.1 also demonstrates that on the torus, a maximal complete $1$-system is unique up to the action of $\mbox{Mod}(S_{1})$. That this is also the case for maximal complete $1$-systems on $S_{2}$ is addressed in \cite{Mal-Riv-Ther}. This motivates Malestein, Rivin, and Theran to pose the following question: 

\begin{question} For $g>2$, are maximal complete $1$-systems unique up to the action of the mapping class group $\mbox{Mod}(S_{g})$?
\end{question}

The main result of this paper answers this question in the negative; specifically, we prove:

\begin{theorem} For $g>3$, there exists at least $2$ distinct mapping class group orbits of maximal complete $1$-systems. 

\end{theorem}
The proof of Theorem 1.2 is completed by explicitly constructing two distinct orbits on $S_{3}$, and then attaching handles in a prescribed way to extend both orbits to higher genus surfaces. The construction for $S_{3}$ makes use of a particular gluing pattern on a $20$-gon; when the edges are glued together via the prescribed pattern, a genus $3$ surface is obtained and the edges of the $20$-gon project down to a pair of simple closed curves intersecting $5$ times. 

In general, a collection of curves $\Gamma=\left\{\gamma_{i}\right\}_{i\in I}$ on $S_{g}$ is said to $\textit{fill}$ $S_{g}$ if $S_{g}\setminus \Gamma$ is a disjoint union of topological disks. It is well known that for any $g$ there exists a pair of simple closed curves which fill. Such filling pairs have been studied extensively; for instance, one of the first examples of pseudo-Anosov mapping classes came from considering compositions of Dehn twists around filling pairs (see \cite{Far-Mar}). 

One is then led to ask the following simple question:

\begin{question}
What is the minimum number of times a pair of simple closed curves on $S_{g}$ must intersect in order to fill $S_{g}$? 
\end{question}

In a forthcoming paper, (joint with Shinnyih Huang), we prove \cite{A.-Huang}

\begin{theorem}  Let $\left\{\alpha,\beta\right\}$ be a pair of simple closed curves which fills $S_{g}$. Then $i(\alpha,\beta)\geq 2g-1$. Furthermore, this bound is sharp for $g\neq 2$. If $g=2$, then $i(\alpha,\beta)\geq 4$. 

\end{theorem}

The sharpness of this bound involves a combinatorial construction which is not directly relevant to this paper. However, we include a proof of the lower bound because it is elementary and it establishes an important correspondence between $i(\alpha,\beta)$ and the number of disks in the complement $S_{g}\setminus (\alpha\cup \beta)$.  \vspace{2 mm}

$\textit{proof of lower bound}$: If $\left\{\alpha,\beta\right\}$ fill, then we can take $\alpha\cup \beta$ to be a $1$-skeleton of a cell decomposition of $S_{g}$, where the vertices are the intersections and the edges are the arcs of $\alpha$ and $\beta$ which run between intersection points. We remark that there are twice as many edges as there are vertices. Letting $D$ denote the number of disks in $S_{g}\setminus (\alpha\cup\beta)$, 
\[ \chi(S_{g})=2-2g= i(\alpha,\beta)-2i(\alpha,\beta)+D\]
\[\Rightarrow 2-2g=D-i(\alpha,\beta)\Rightarrow i(\alpha,\beta)=2g-2+D\]
Since $D\geq 1$, it follows that $i(\alpha,\beta)\geq 2g-1$. $\Box$ \vspace{2 mm}

Therefore, a simple closed filling pair having a minimal number of intersections is equivalent to its complement being connected. Thus cutting along such a pair of curves produces a single polygon with $8g-4$ sides. One expects that a minimally intersecting filling pair should `barely' fill the surface, in that there are many simple closed curves intersecting the pair only once; this was the motivation for considering the gluing patterns associated to minimally intersecting filling pairs to study large complete $1$-systems.

The organization of the paper is as follows. In section $2$, we demonstrate the existence of a complete maximal $1$-system on $S_{g}$ and we provide a necessary condition for any other maximal complete $1$-system to be equivalent to this one. In section $3$, we complete the construction for $S_{3}$, and in section $4$, we extend this construction to all higher genera. 

The author would like to thank Yair Minsky, W. Patrick Hooper, and Ian Biringer for numerous helpful conversations. 

\section{Existence of Maximal Complete $1$-Systems}
In this section we briefly demonstrate that there exists a complete $1$-system with $2g+1$ members on $S_{g}$ for all $g\geq 1$. Indeed, $S_{g}$ can be obtained by identifying the opposite sides of a $4g+2$-gon $P$. Then the arcs connecting opposite sides of $P$ project down to simple closed curves, any two of which intersect exactly once.

\begin{figure}[H]
\caption{ A Maximal Complete $1$-system on a genus $2$ surface}
\centering
	\includegraphics[width=4in]{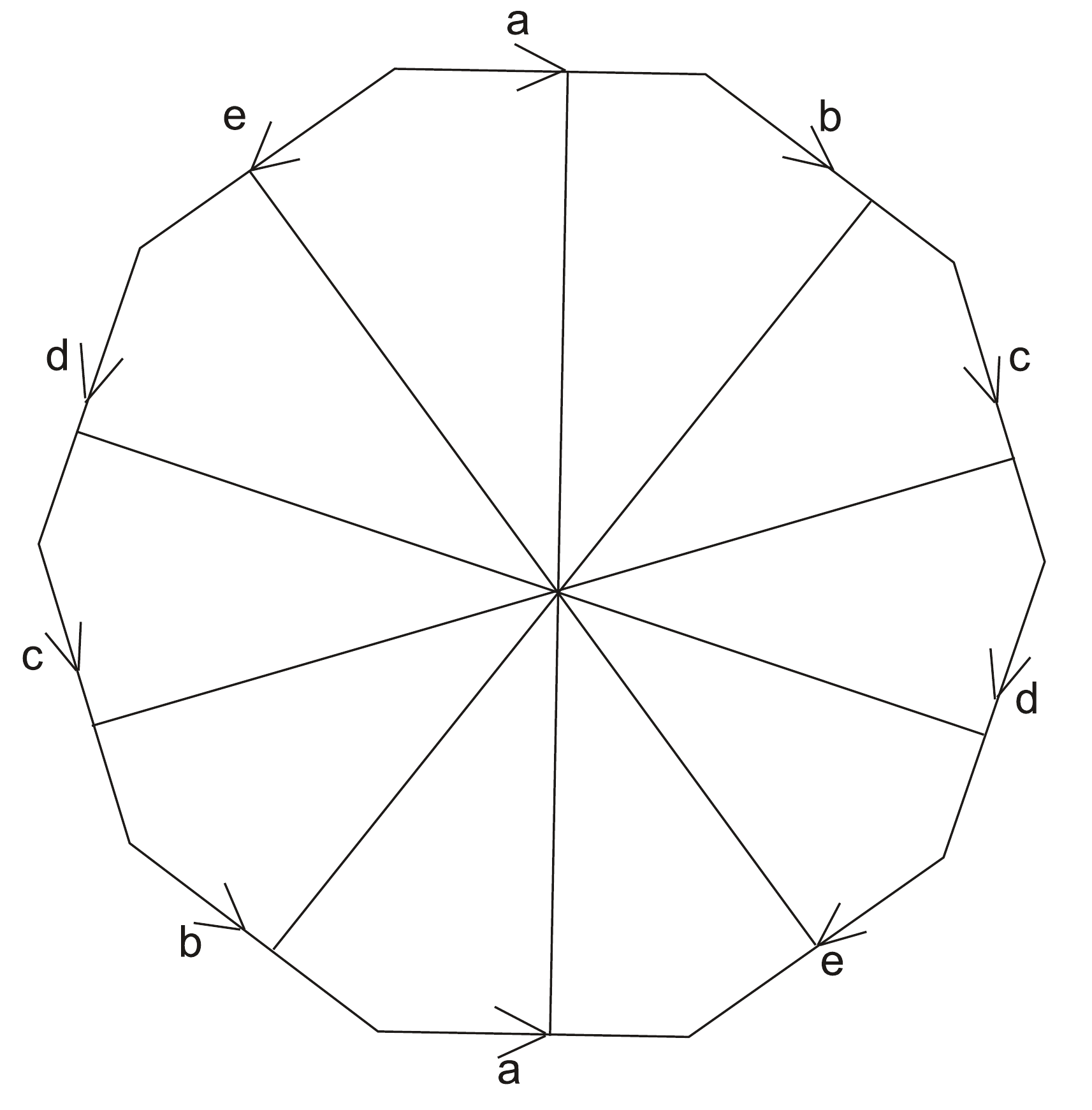}
\end{figure}

Denote this complete maximal $1$-system by $X(g)$ (so that $X(2)$ is pictured above). By construction, it is clear that all intersections in $X(g)$ can be taken to occur at the same point. This property characterizes a class of maximal complete $1$-systems, as the following proposition makes precise. 

\begin{proposition} Let $\Omega$ be a complete maximal $1$-system on $S_{g}$ having the property that up to isotopy, all intersections occur at the same point. Then $\Omega$ and $X(g)$ are equivalent modulo the action of $\mbox{Mod}(S_{g})$.
\end{proposition}

\begin{proof} Up to ambient isotopy, thinking of $\Omega= \left\{\omega_{1},...,\omega_{2g+1}\right\}$ as a graph on $S_{g}$, $\Omega$ has one vertex and $2g+1$ edges. Denoting by $N(\Omega)$ a small regular neighborhood of $\Omega$, one can trace around $\partial N(\Omega)$ to discover that $S_{g}\setminus \Omega$ has exactly two connected components, which we denote by $A$ and $B$ (see Figure $2$ below).

\begin{figure}[H]
\caption{The local picture at the only vertex of $\Omega(2)$.  The complement of $\Omega(g)$ in $S_{g}$ has two connected components, and each curve in $\Omega(g)$ is on the boundary of both $A$ and $B$.}
\centering
\includegraphics[width=4in]{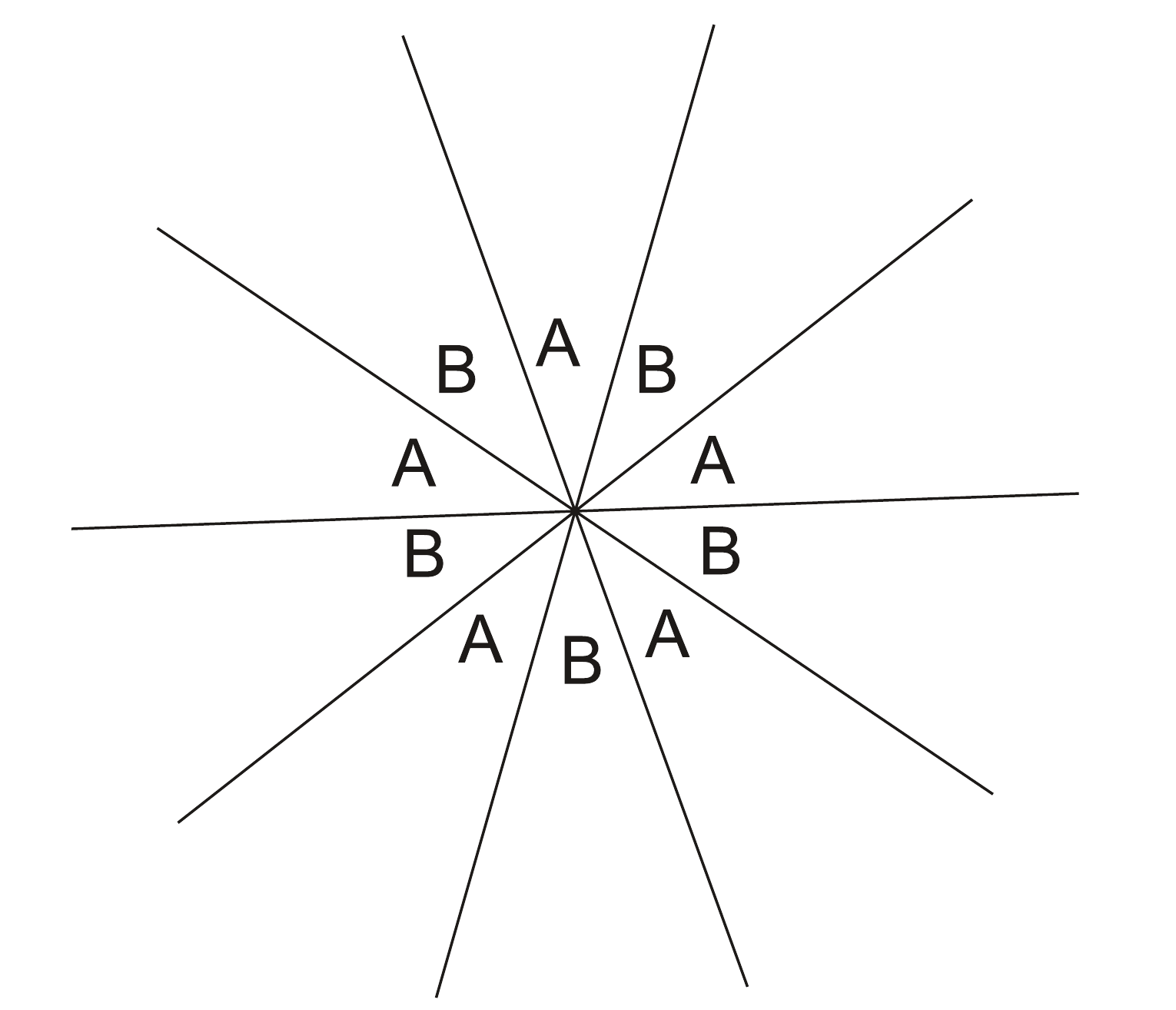}
\end{figure}

We claim that $A$ and $B$ are both disks; indeed, 
\[ \chi(S_{g})=2-2g=\chi(\Omega)+\chi(A)+\chi(B)\]
\[=-2g+ \chi(A)+\chi(B)\Rightarrow \chi(A)=\chi(B)=1 \]
 Pick points $a\in A$ and $b\in B$; then for each $j$ there exists an arc $\gamma_{j}$ such that $i(\gamma_{j},\omega_{i})=\delta_{i}^{j}$, and $\gamma_{j}$ has endpoints $a$ and $b$. 
Letting $x$ denote the intersection point of $\Omega$. $\left\{x,\Omega,\left\{A,B\right\}\right\}$ is a cellular decomposition of $S_{g}$, and $\left\{\left\{a,b\right\},\left\{\gamma_{1},...,\gamma_{2g+1}\right\}\right\}$ can be extended to a dual decomposition with one $2$-cell containing $x$. Thus cutting along $\gamma_{1},...,\gamma_{2g+1}$ yields a polygon with $4g+2$ edges, and $\omega_{1},...,\omega_{2g+1}$ necessarily lift to line segments connecting opposite sides.

\end{proof}
 \vspace{100 mm}
\section{Genus $3$ Case}
Pictured below is a gluing pattern on a $20$-gon which yields the closed orientable genus $3$ surface $S_{3}$. 

\begin{figure}[H]
\centering
	\includegraphics[width=4in]{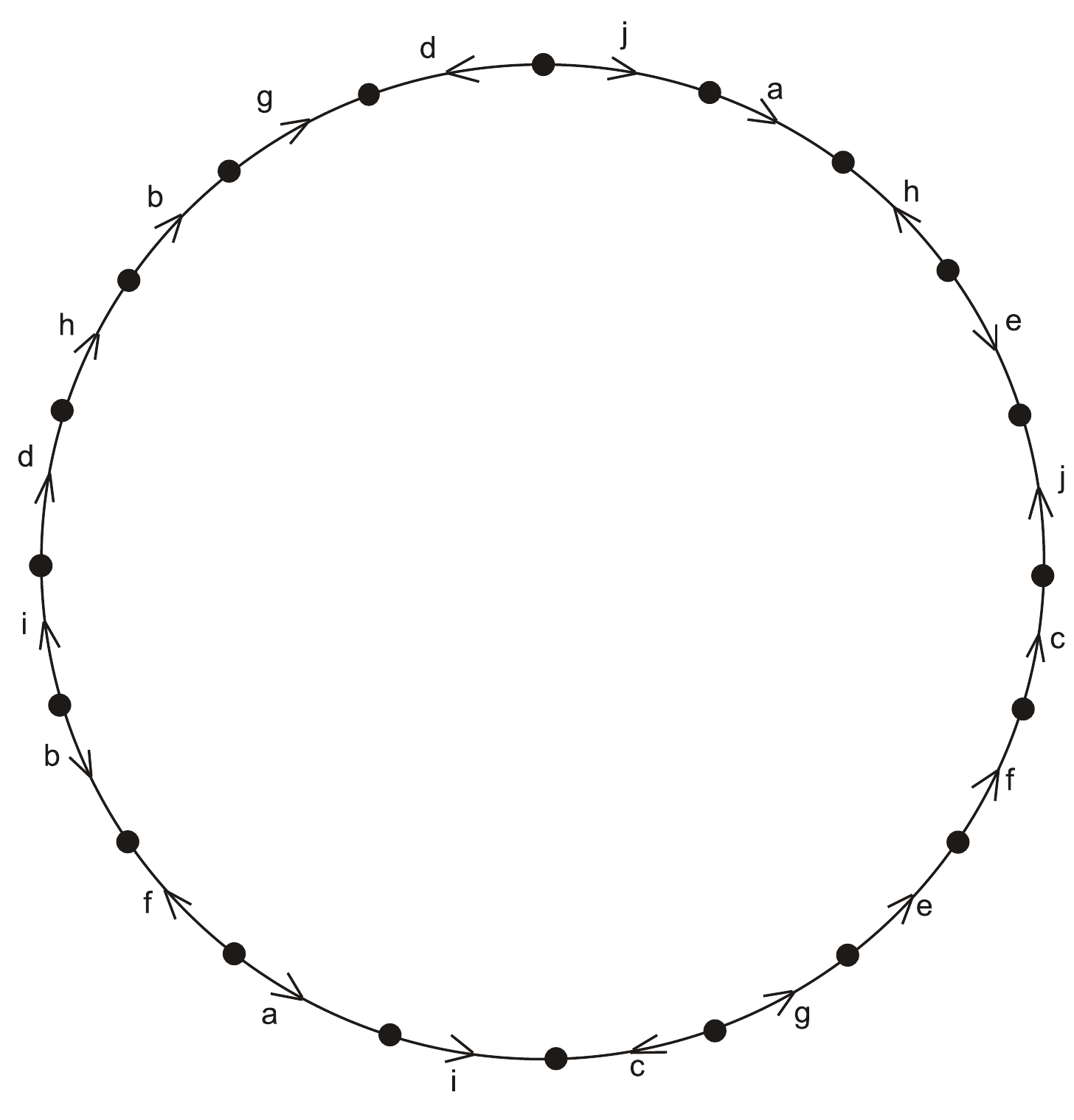}
	\caption{A Gluing Pattern on a $20$-gon which produces a genus $3$ surface. The edges alternate between projecting down to one of two simple closed curves.}

\end{figure}

The significance of the lettering is that edges a,b,c,d,e project down to arcs of one simple closed curve $\alpha$ in that order, and edges f,g,h,i,j project down to arcs of another simple closed curve $\beta$ in that order; $\alpha$ and $\beta$ intersect $5$ times. Note that there are $5$ distinct equivalence classes of vertices, each containing $4$ vertices of the above polygon. These $5$ classes of vertices correspond to the $5$ intersections between $\alpha$ and $\beta$. We will use this gluing pattern to construct a maximal complete $1$-system $\Omega(3)$ which is distinct from $X(3)$, modulo the action of $\mbox{Mod}(S_{3})$.

 \vspace{150 mm}

We begin the construction by considering the three curves pictured below- one in red, one blue, and the third black. The three curves pairwise intersect exactly once and thus form a complete $1$-system. Note that attempting to isotope the curves so that all three intersections occur at the same place may not, a priori, keep all three curves simple. Indeed, we claim that these three curves cannot be isotoped to make all three intersections coincide.

\begin{lemma} The three curves pictured below cannot be isotoped so as to make all three pairwise intersections coincide. 
\end{lemma}

\begin{figure}[H]
\centering
	\includegraphics[width=4 in]{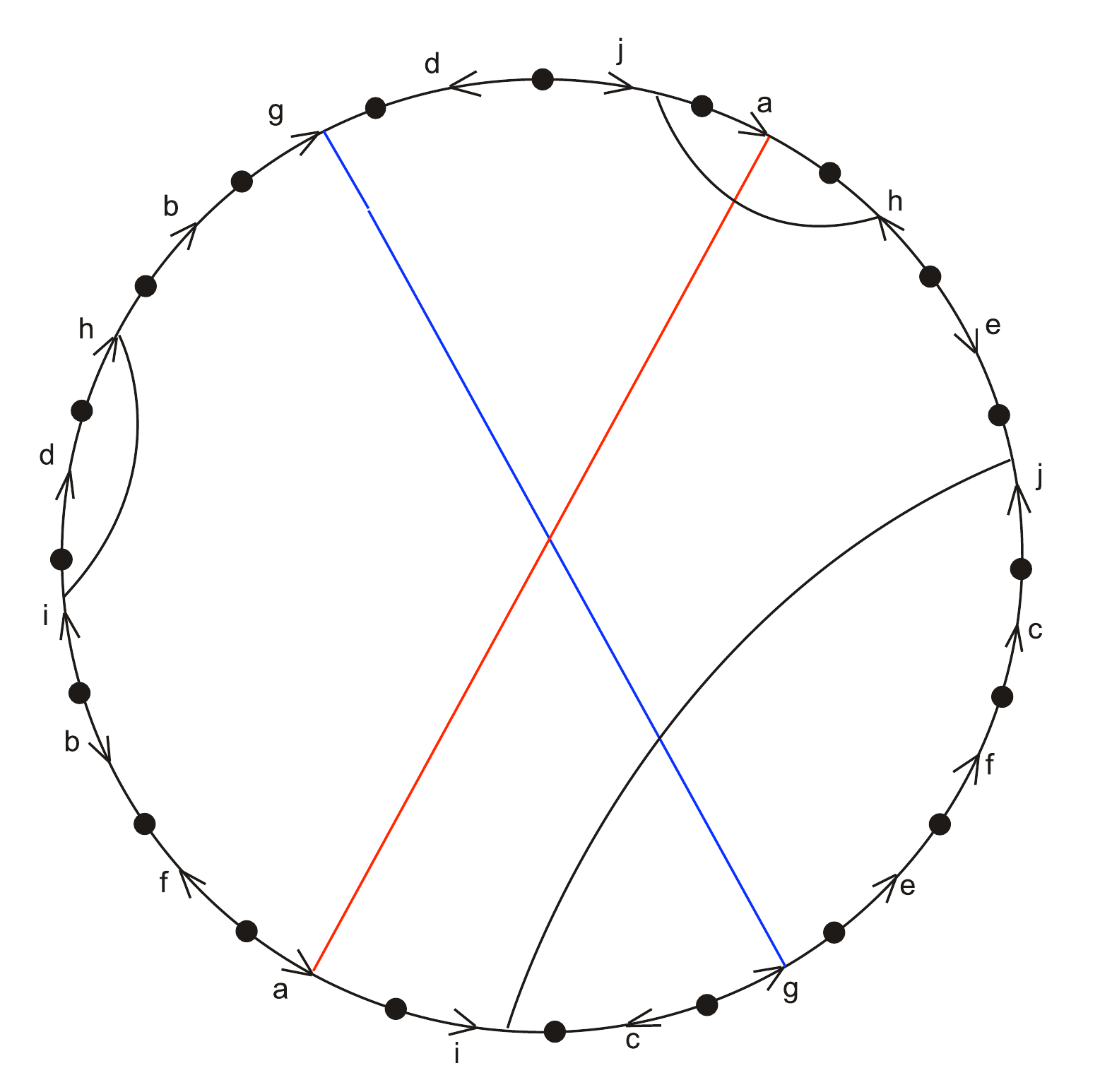}
	\caption{Three curves forming a complete $1$-system. Attempting to isotope the curves so that all intersections occur at the same point may result in one of the curves having self intersection. }
\end{figure}

\begin{proof}
Let $\Sigma\subset S_{3}$ be the subsurface consisting of a regular neighborhood of the union of the three curves. Consider the four curves pictured below in Figure $5$; the red and blue curves are the curves of the same colors in Figure $4$; let $\Gamma< \pi_{1}(\Sigma)$ be the subgroup of the fundamental group of $\Sigma$ generated by these four curves. Then we will show that $i_{\ast}(\Gamma)$ is a rank $4$ subgroup of $\pi_{1}(S_{3})$, where $i_{\ast}:\pi_{1}(\Sigma)\rightarrow \pi_{1}(S_{3})$ is the homomorphism induced by the inclusion $i:\Sigma \hookrightarrow S_{3}$.  This will imply the conclusion of Lemma $3.1$, for if the three curves in figure $4$ could be isotoped so that all three intersections occur at the same place, it would imply that there is an isotopy of $S_{3}$ sending $\Sigma$ to a regular neighborhood of a bouquet of three circles. The rank of $i_{\ast}(\pi_{1}(\Sigma))$ is unaffected by this isotopy, and therefore if such an isotopy exists, $\mbox{rank}(i_{\ast}(\pi_{1}(\Sigma)))\leq 3$.

\begin{figure}[H]
\centering
	\includegraphics[width=3.5in]{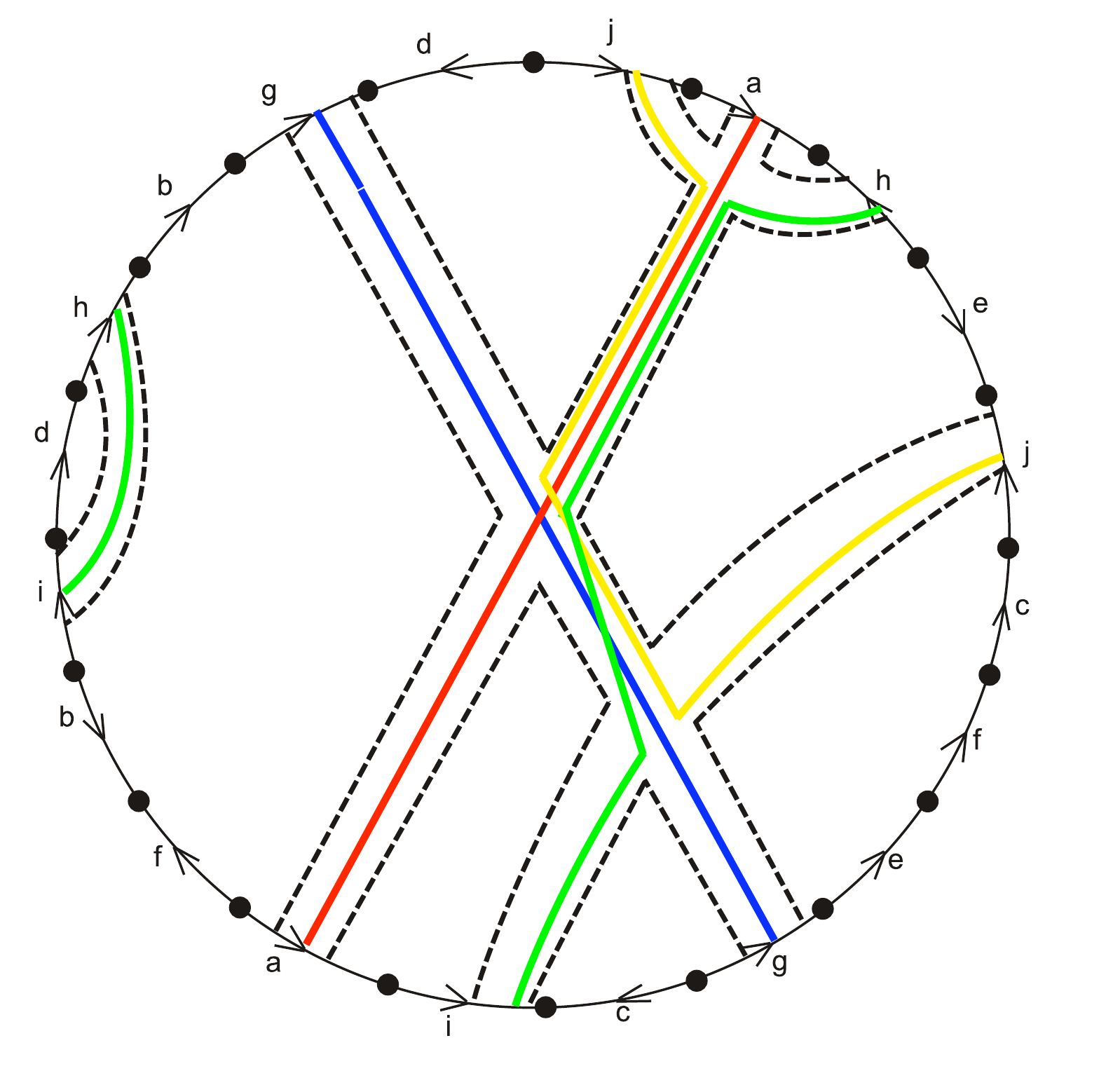}
	\caption{Four simple closed curves in $\pi_{1}(\Sigma)$ colored red, blue, yellow and green. The boundary of $\Sigma$ is drawn in dashed lines. }
\end{figure}

The path consisting of the concatenated edges b,c,d,e comprise a spanning tree for $\alpha\cup\beta$; thus we can contract them to obtain the following presentation for $\pi_{1}(S_{3})$:

\[ \pi_{1}(S_{3})= \left\langle a,h,i,j,g,f | ah^{-1}j^{-1}f^{-1}g^{-1}i^{-1}a^{-1}fihgj=1 \right\rangle \]

We abelianize $\pi_{1}(S_{3})$ to obtain $\mathbb{Z}^{6}$ by identifying the $n^{th}$ generator in the above presentation with the $n^{th}$ row vector of the $6\times6$ identity matrix; note that this is a valid abelianization because each generator and its inverse appear exactly once in the relator for the above presentation. Then we can read off the images of the four curves pictured under the abelianization map; in the order of red, blue, yellow, green, the images are 

\[ (0,-1,-1,-1,-1,-1); (1,-1,0,0,0,-1); (1,-1,0,0,0,0); (0,0,0,-1,-1,-1) \]

These four vectors are linearly independent over $\mathbb{R}$ and therefore also over $\mathbb{Z}$. Hence $i_{\ast}(\pi_{1}(\Gamma))$ has rank $4$. 
\end{proof}

Alternatively, as was pointed out by Ian Biringer, this conclusion can be reached by observing that $\chi(\Sigma)=-3$, $\partial \Sigma$ has three connected components, and $S_{3}\setminus \Sigma$ is connected. Thus $\Sigma$ is a thrice punctured torus, and by examining the long exact sequence for homology of the pair $(S_{3},\Sigma)$ one finds that $H_{1}(\Sigma)$ injects into $H_{1}(S_{3})$. 

To complete the construction of $\Omega(3)$, it suffices to extend the three curves pictured in Figure $4$ to a maximal complete $1$-system by adding $4$ more curves. Such a maximal complete $1$-system would have the property that it cannot be isotoped so as to make all of the $21$ intersections coincide, and therefore can not be equivalent to $X(3)$ modulo the action of $\mbox{Mod}(S_{3})$. Indeed, such a complete $1$-system exists and it is pictured below in Figure $6$; the three curves of Figure $4$ are colored according to their colors in Figure $4$. 

\begin{figure}[H]
\centering
	\includegraphics[width=3.5in]{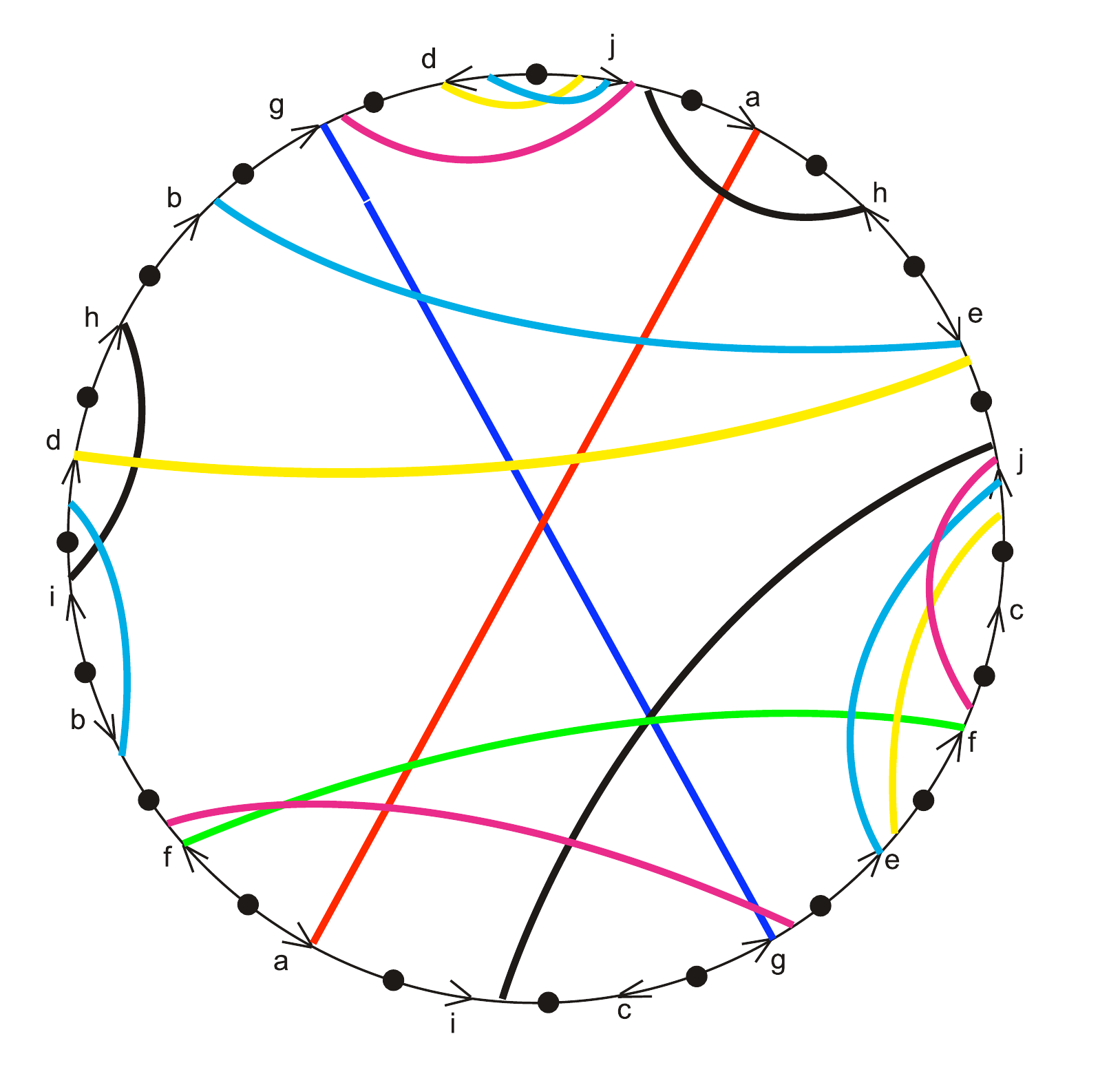}
\caption{A maximal complete $1$-system containing the three curves depicted in Figure $4$ as a sub-complete $1$-system.}
\end{figure}

Note that no two of the above curves are isotopic because each intersects the other exactly once transversely; for the same reason no curve pictured above can be nulhomotopic (or even nulhomologous). This completes the construction of $\Omega(3)$, and proves Theorem 1.2 in the case $g=3$.  
\section{Higher Genus}

To construct $\Omega(4)$ we make the following observation: there exists a simple arc $\gamma(3)$ which intersects each member of $\Omega(3)$ exactly once transversely (see Figure $7$). 

\begin{figure}[H]
\centering
	\includegraphics[width=3.5in]{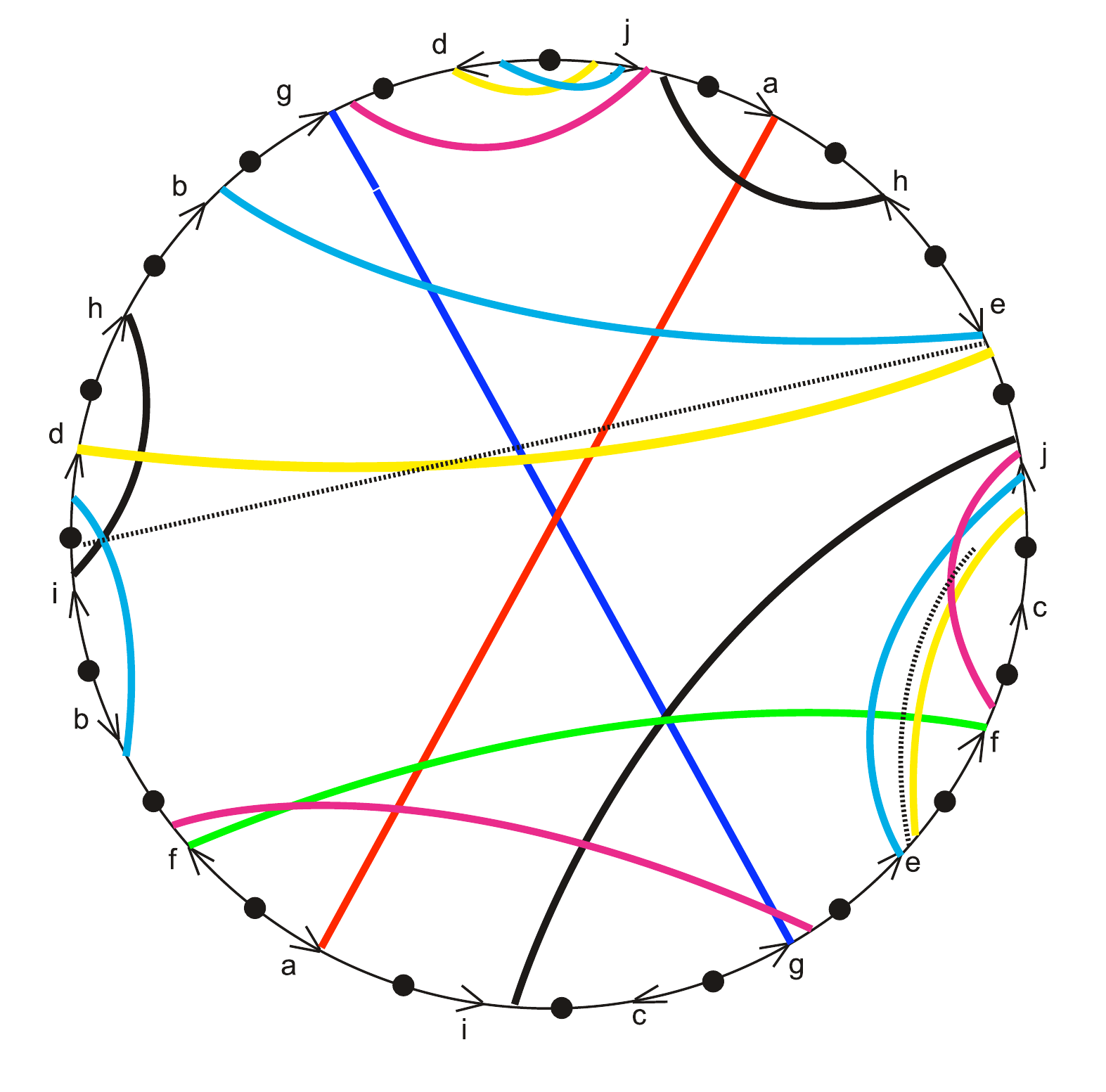}
\caption{There exists a simple arc $\gamma(3)$ (pictured above as a dashed line) transversely intersecting each curve in $\Omega(3)$ exactly once.}
\end{figure}

Then we can cut out two open disks $D_{1},D_{2}$ with disjoint closures, such that the boundary of each contains one endpoint of $\gamma(3)$ and such that the closure of each is disjoint from $\Omega(3)$. Take two parallel copies $\gamma_{1}(3),\gamma_{2}(3)$ of $\gamma(3)$, and glue an annulus $A$ to $S_{3}\setminus(D_{1}\cup D_{2})$ along its two boundary components. By `parallel', we mean an arc isotopic to $\gamma(3)$ via an isotopy which keeps the endpoints on $\partial D_{1}$ and $\partial D_{2}$. Let $\kappa_{1}$ be a proper simple arc on $A$ (proper meaning that it has one endpoint on each boundary component of $A$) intersecting the core curve of $A$ exactly once, and let $\kappa_{2}$ be the image of $\kappa_{1}$ under a Dehn twist around the core curve of $A$. Then the gluing of $A$ to $S_{3}\setminus (D_{1}\cup D_{2})$ can be done in such a way so that the endpoints of $\kappa_{j}$ are identified with the endpoints of $\gamma_{j}$, $j=1,2$. 
 
Denote by $\psi_{j}(3)$ the simple closed curve comprised of the concatenation of $\kappa_{j}$ with $\gamma_{j}$. Then $\psi_{j}(3)$ intersects each member of $\Omega(3)$ exactly once, and by construction $\psi_{1}(3)$ and $\psi_{2}(3)$ intersect exactly once. Therefore $\left\{\psi_{1}(3),\psi_{2}(3),\Omega(3)\right\}=\Omega(4)$ is a maximal complete $1$-system on $S_{4}$. No curves on $S_{3}$ become isotopic as a result of removing $D_{1}$ and $D_{2}$ and gluing on $A$; therefore the three curves of Figure $4$ still generate a rank $4$ subgroup of $\pi_{1}(S_{4})$, so the intersections of $\Omega(4)$ cannot be taken to all coincide. Thus $\Omega(4)$ is in a distinct mapping class group orbit from $X(4)$. 

This process can be repeated to construct $\Omega(5)$ which will be distinct from $X(5)$, because again we can find an arc $\gamma(4)$ which transversely intersects every member of $\Omega(4)$ exactly once. To see this, simply consider a simple arc $\theta(4)$ on $A$ with one endpoint on one of the boundary components of $A$, and the other endpoint being the intersection point of $\psi_{1}(3)$ and $\psi_{2}(3)$. We concatenate this arc with a parallel copy of $\gamma(3)$ to obtain $\gamma(4)$, and then we proceed as before- gluing on another annulus to obtain $S_{5}$ and $\Omega(5)$. 

We proceed inductively; assuming that there exists $\gamma(n-1)$, a simple arc transversely intersecting each member of $\Omega(n-1)$ exactly once, by simply concatenating $\gamma(n-1)$ with $\theta(n-1)$ we obtain $\gamma(n)$, a simple arc transversely intersecting each member of $\Omega(n)$. This then allows us to construct $\Omega(n+1)$, which is in a distinct mapping class group orbit from that of $X(n+1)$. This concludes the proof of Theorem 1.2.

\end{document}